\documentclass[a4paper]{amsart}
\usepackage[english]{babel}
\usepackage[utf8]{inputenc}
\usepackage{amsmath, amssymb, amsthm}

\newtheorem{theorem}{Theorem}
\newtheorem{definition}{Definition}
\theoremstyle{plain} 
\newtheorem{corollary}{Corollary}

\newtheorem{lemma}{Lemma}
\newtheorem{proposition}{Proposition}

\newtheorem{example}{Example}

\newtheorem{remark}{Remark}

\newtheorem{notation}{Notation}

\newcommand{\N}{{\mathbb{N}}}
\newcommand{\C}{{\mathbb{C}}}

\newcommand{\Z}{{\mathbb{Z}}}
\newcommand{\R}{{\mathbb{R}}}

\title{The Pascal matrix in the multivariate Riordan group}

\author{Helena Cobo}

\begin{document}

\begin{abstract}
We generalize the concept of Pascal matrices to matrices associated with sets of points $R\subseteq\Z^n_{\geq 0}$ by considering multidimensional binomial coefficients as entries. We study their properties and prove that the infinite matrix associated with the set $R=\Z^n_{\geq 0}$ is in fact an element of the multivariate Riordan group.
\end{abstract}

\maketitle

\section*{Introduction}
Though the Pascal triangle is present in mathematical texts since ancient times, it is not till recently that Pascal matrices were studied with some detail. Probably the first reference devoted to Pascal matrices is \cite{Lunnon}. A good introduction to this topic is \cite{ES18}, see also the references \cite{CV}, \cite{AcetoTrigiante}, \cite{BL} or \cite{Hiller}, where Pascal matrices are presented from different points of view.

\vspace{2mm}

The (classical) Pascal matrices are matrices whose entries are binomial coefficients. More precisely:
\begin{definition}
For $n\in\N$, the squared matrices
\[L_n=\left(\binom{i}{j}\right)_{0\leq i,j\leq n},\ \ U_n=L_n^T\ \mbox{ and }S_n=\left(\binom{i+j}{i}\right)_{0\leq i,j\leq n}\in\mathcal M_{(n+1)\times(n+1)}(\N)\]
are the lower-triangular, upper-triangular and symmetric Pascal matrices of order $n+1$, respectively.

We can also consider the respective matrices of infinite order. In fact, the infinite lower-triangular matrix
\[\left(\binom{i}{j}\right)_{0\leq i,j<\infty}\]
is called the (classical) Pascal matrix. It is also known as the binomial matrix.
\label{Pascalcl}
\end{definition}

There are several generalizations of this notion in the literature. For instance, in \cite{CV} (see also \cite{Z}) the generalized Pascal matrix of order $n+1$ is defined as
\[\left(x^{i-j}\binom{i}{j}\right)_{0\leq i,j\leq n}\]
considering its entries in $\Z[x]$. See also \cite{Barry}, \cite{WW} and \cite{MikCheon} for more possible generalizations.

Another generalizations appear by considering alternative definitions of the Pascal triangle, yielding for instance to the Hoggatt triangles (see \cite{FA}). See \cite{Bur} for yet another type of generalized Pascal matrices, with a more sophisticated definition of binomial coefficients.

In this paper we present an alternative generalization of Pascal matrices. We define the multivariate Pascal matrices as matrices whose entries are multidimensional binomial coefficients.

We came across the multivariate Pascal matrix by studying linear systems defined by matrices whose entries are Stirling polynomials as defined in \cite{Cobo1}. These linear systems appear when computing b-functions with respect to weights of certain holonomic ideals. The key to solve such linear systems is the decomposition given in Theorem \ref{Decomp}.

We introduce the Pascal matrices in the first section and study some of its properties in the second section. The last section is devoted to the multivariate Riordan group. This group is defined in terms of $n+1$ formal power series in $\C[[x_1,\ldots,x_n]]$ as a generalization of the classical (univariate) Riordan group, which is defined in terms of two formal power series in $\C[[x]]$. We prove that the infinite Pascal matrix is an element of the multivariate Riordan group.

\vspace{5mm}

{\bf Acknowledgment:} I want to thank Sofía Cobo for her careful reading of a first draft of this paper. I also want to thank the referee for his sensible comments, especially those concerning the definition of the multivariate Riordan group.

\section{Multivariate Pascal matrices}
Before defining the Pascal matrices let us set some notations.
\begin{notation}
For any ${\bf k}=(k_1,\ldots,k_n)\in\Z^n_{\geq 0}$ we use the standard notations
\[\begin{array}{l}
|{\bf k}|=k_1+\cdots+k_n,\\
\\
{\bf k}!=k_1!\cdots k_n!.\\
\end{array}\]
For ${\bf i},{\bf k}\in\Z^n_{\geq 0}$, we write ${\bf i}\leq{\bf k}$ to indicate the partial order on $\Z^n_{\geq 0}$ defined as
\[0\leq i_j\leq k_j\ \mbox{ for }\ 1\leq j\leq n.\]
By ${\bf i}<{\bf k}$ we mean ${\bf i}\leq{\bf k}$ and ${\bf i}\neq{\bf k}$.

The set $\{{\bf e_1},\ldots,{\bf e_n}\}$ denotes the standard basis of $\R^n$. We denote by ${\bf 0}$ the origin $(0,\ldots,0)$ of $\Z^n$, and by ${\bf 1}$ the vector $(1,\ldots,1)\in\Z^n$.

In this context the Kronecker delta of two vectors is defined as
\[\delta_{{\bf k}{\bf k'}}=\delta_{k_1k_1'}\cdots\delta_{k_nk_n'}.\]

\vspace{3mm}

For variables $(x_1,\ldots,x_n)$ we denote by ${\bf x}^{\bf k}$ the monomial $x_1^{k_1}\cdots, x_n^{k_n}$.

\end{notation}

\vspace{3mm}

The generalization of the Pascal matrices lies in the generalized binomials
\[\binom{\bf k}{\bf i}=\binom{k_1}{i_1}\cdots\binom{k_n}{i_n}, \mbox{ for }{\bf i}\leq{\bf k}.\]
The combinatorial interpretation of these numbers is the following. Suppose we have $n$ sets $S_1,\ldots,S_n$ with cardinal $|S_i|=k_i$. Then for ${\bf i}\leq{\bf k}$, the number $\binom{\bf k}{\bf i}$ equals the number of ways in which we can select $i_j$ different elements from the set $S_j$, for $j=1,\ldots,n$.

\vspace{3mm}

Next lemma illustrates the analytic interpretation of the multidimensional binomial coefficients.

\begin{lemma}
The multidimensional binomials $\binom{\bf k}{\bf k'}$ appear in the expansions:
\begin{enumerate}
\item[(i)]
\[(1+x_1)^{k_1}\cdots(1+x_n)^{k_n}=\sum_{{\bf k'}\leq{\bf k}}\binom{\bf k}{\bf k'}{\bf x}^{\bf k'},\]
\item[(ii)]
\[\frac{1}{(1-x_1)^{k_1}}\cdots\frac{1}{(1-x_n)^{k_n}}=\sum_{{\bf k'}\in\Z_{\geq 0}^n}\binom{{\bf k}+{\bf k'}-{\bf 1}}{\bf k'}{\bf x}^{\bf k'},\]
\item[(iii)]
\[\frac{1}{{\bf k}!}\frac{\partial^{|{\bf k}|}}{\partial^{k_1}x_1\cdots\partial^{k_n}x_n}\big({\bf x}^{\bf k'}\big)=\binom{\bf k'}{\bf k}{\bf x}^{{\bf k'}-{\bf k}}.\]
\end{enumerate}
\label{anBin}
\end{lemma}

{\em Proof.}
Equation in (i) is a straightforward generalization of the binomial identity $(1+x)^n=\sum_{k=0}^n\binom{n}{k}x^k$, while equation in (ii) generalizes
the identity $\frac{1}{(1-x)^n}=\sum_{k=0}^\infty\binom{n+k-1}{k}x^k$ (see for instance \cite{Riordanbook}). Finally (iii) is a straightforward generalization of $\frac{\partial x^a}{\partial x^b}=\frac{a!}{(a-b)!}x^{a-b}$.
\hfill$\Box$

\vspace{3mm}

The following properties will be very useful.

\begin{lemma}
Let ${\bf k},{\bf j}\in\Z^n_{\geq 0}$ with ${\bf j}\leq{\bf k}$, then
\begin{equation}
\sum_{{\bf j}\leq{\bf i}\leq{\bf k}}\binom{\bf k}{\bf i}\binom{\bf i}{\bf j}(-1)^{|{\bf i}|-|{\bf j}|}=\delta_{{\bf j}{\bf k}}.
\label{eqbinom}
\end{equation}
Moreover, for ${\bf k},{\bf r}\in\Z^n_{\geq 0}$, the Vandermonde identity extends to
\begin{equation}
\sum_{{\bf k'}\in\Z^n_{\geq 0},\ {\bf k'}\leq{\bf k}}\binom{\bf r}{\bf k'}\binom{\bf k}{\bf k'}=\binom{{\bf r}+{\bf k}}{\bf k},
\label{eqVand}
\end{equation}
where we convey that
\[\binom{\bf k}{\bf k'}=0\ \ \mbox{ if }\ {\bf k'}\not\leq{\bf k}.\]

\end{lemma}

{\em Proof.} These identities are generalizations of the well known identities (see any book on  combinatorics such as \cite{Riordanbook}):
\[\begin{array}{c}
\sum_{i=j}^k\binom{k}{i}\binom{i}{j}(-1)^{i-j}=\delta_{jk},\\
\\
\sum_{k'=0}^k\binom{r}{k'}\binom{k}{k'}=\binom{r+k}{k}.\\
\end{array}\]
These are the one-dimensional versions of the identities we have to prove. We prove the generalization of the first identity:
\[\begin{array}{ll}
\sum_{{\bf j}\leq{\bf i}\leq{\bf k}}\binom{\bf k}{\bf i}\binom{\bf i}{\bf j}(-1)^{|{\bf i}|-|{\bf j}|} & =\sum_{{\bf j}\leq{\bf i}\leq{\bf k}}\binom{k_1}{i_1}\cdots\binom{k_n}{i_n}\binom{i_1}{j_1}\cdots\binom{i_n}{j_n}(-1)^{i_1+\cdots+i_n-j_1-\cdots-j_n}\\
\\
 & =\Big(\sum_{i_1=j_1}^{k_1}\binom{k_1}{i_1}\binom{i_1}{j_1}(-1)^{i_1-j_1}\Big)\cdots\Big(\sum_{i_n=j_n}^{k_n}\binom{k_n}{i_n}\binom{i_n}{j_n}(-1)^{i_n-j_n}\Big)\\
 \\
  & =\delta_{j_1k_1}\cdots\delta_{j_nk_n}\\
\end{array}\]
The generalization of the Vandermonde identity goes analogously:
\[\begin{array}{ll}
\sum_{{\bf k'}\in\Z^n_{\geq 0},\ {\bf k'}\leq{\bf k}}\binom{\bf r}{\bf k'}\binom{\bf k}{\bf k'} & =\sum_{0\leq k_i'\leq k_i}\binom{r_1}{k_1'}\cdots\binom{r_n}{k_n'}\binom{k_1}{k_1'}\cdots\binom{k_n}{k_n'}\\
\\
 & =\Big(\sum_{k_1'=0}^{k_1}\binom{r_1}{k_1'}\binom{k_1}{k_1'}\Big)\cdots\Big(\sum_{k_n'=0}^{k_n}\binom{r_n}{k_n'}\binom{k_n}{k_n'}\Big)\\
 \\
  & =\binom{r_1+k_1}{k_1}\cdots\binom{r_n+k_n}{k_n}\\
\end{array}\]
 \hfill$\Box$

\vspace{3mm}

\begin{notation}
We denote by
 \[{\rm diag}\big(a_1,\ldots,a_n\big)\]
the $n\times n$ matrix with entries $a_1,\ldots,a_n$ in the diagonal and zero elsewhere.
In particular, $I_n$ is the identity matrix ${\rm diag}(1,\ldots,1)$ of order $n$.
\end{notation}

\vspace{3mm}

Let $<_T$ be a total order on $\Z^n_{\geq 0}$ and $R\subseteq\Z^n_{\geq 0}$ a finite set of points.
 If $<_T$ orders the points in $R$ as
\[R=\{{\bf k_1},\ldots,{\bf k_r}\}\]
we can construct the multidimensional binomial triangle associated with the set $R$ as:
\[\begin{array}{ccccc}
& & \binom{\bf k_1}{\bf k_1} & & \\
\\
 & \binom{\bf k_2}{\bf k_1} & & \binom{\bf k_2}{\bf k_2} & \\

\\
 &  & \vdots & & \\
\\
\binom{\bf k_r}{\bf k_1} & \binom{\bf k_r}{\bf k_2} & \cdots & \binom{\bf k_r}{\bf k_{r-1}} & \binom{\bf k_r}{\bf k_r}\\
\end{array}\]

\vspace{3mm}

The idea is to represent this triangle as a squared matrix. We could index the matrices by $i$ and $j$ with $1\leq i,j\leq r$, but it is convenient to index them by the vectors ${\bf k_i},{\bf k_j}\in\Z^n_{\geq 0}$. As in the classical case, there are three ways to represent Pascal matrices.

\vspace{3mm}

\begin{definition}
Let $R=\{{\bf k_1},\ldots,{\bf k_r}\}$ be a finite set of points in $\Z^n_{\geq 0}$. We define the $r\times r$ matrices
\[L_R=\Big(\ell_{{\bf k_i},{\bf k_j}}\Big)\ \mbox{ and }U_R=L_R^t\]
where $\ell_{{\bf k_i},{\bf k_j}}=\binom{\bf k_i}{\bf k_j}$ and the rows and columns of the matrices are ordered by a total order $<_T$ on $\Z^n_{\geq 0}$. Moreover we define the symmetric matrix
\[S_R=\Big(s_{{\bf k_i},{\bf k_j}}\Big)\]
with $s_{{\bf k_i},{\bf k_j}}=\binom{{\bf k_i}+{\bf k_j}}{\bf k_i}$.
\[S_R=\left(\begin{array}{cccc}
\binom{2{\bf k_1}}{\bf k_1} & \binom{{\bf k_1}+{\bf k_2}}{\bf k_1} & \cdots & \binom{{\bf k_1}+{\bf k_r}}{\bf k_1}\\
\\
\binom{{\bf k_1}+{\bf k_2}}{\bf k_2} & \binom{2{\bf k_2}}{\bf k_2} & \cdots & \binom{{\bf k_2}+{\bf k_r}}{\bf k_2}\\
\\
\vdots & & & \vdots\\
\\
\binom{{\bf k_1}+{\bf k_r}}{\bf k_r} & & \cdots & \binom{2{\bf k_r}}{\bf k_r}\\
\end{array}\right)\]
\end{definition}

\vspace{3mm}

\begin{example}
Consider the ordered set $R=\{(1,0),(0,0),(0,1)\}$. The corresponding Pascal matrix $L_R$ is
\[L_R=\left(\begin{array}{ccc}
1 & 1 & 0\\
0 & 1 & 0\\
0 & 1 & 1\\
\end{array}\right)\]
Obviously we can order the set $R$ in a more convenient way, so that the Pascal matrix $L_R$ is triangular. In this simple example, the choices are $R=\{(0,0),(1,0),(0,1)\}$ or $R=\{(0,0),(0,1),(1,0)\}$.
\end{example}

\begin{proposition}
Let $<_T$ be any total order on $\Z^n$ and let $R\subseteq\Z^n_{\geq 0}$ be a set of points ordered by $<_T$. The matrix $L_R$ is lower-triangular (and hence $U_R$ upper-triangular) if and only if the total order $<_T$ is compatible with the partial order $\leq$, i.e., if ${\bf k}\leq{\bf k'}$ with ${\bf k}\neq{\bf k'}$, then ${\bf k}<_T{\bf k'}$.
\label{lemOrder}
\end{proposition}

{\em Proof.}
Notice that the matrix $L_R$ is lower-triangular if and only if $\binom{\bf k}{\bf k'}=0$ whenever ${\bf k}<_T{\bf k'}$. On the other hand, $\binom{\bf k}{\bf k'}=0$ if and only if there exists $1\leq i\leq n$ such that $k_i<k_i'$.

Hence the matrix $L_R$ is lower-triangular if and only if for all ${\bf k},{\bf k'}\in R$ with ${\bf k}<_T{\bf k'}$ there exists $1\leq i\leq n$ such that $k_i<k_i'$. Or equivalently, for all ${\bf k},{\bf k'}\in R$ with ${\bf k}\geq {\bf k'}$ then ${\bf k}\geq_T{\bf k'}$.
\hfill$\Box$

\vspace{2mm}

\begin{corollary}
Let $R\subseteq\Z_{\geq 0}^n$ be a finite set of points. Then
\[\mbox{det}\big(L_R\big)=\mbox{det}\big(U_R\big)=1\]

\end{corollary}

{\em Proof.} It is enough to prove it for $L_R$. Let $R'$ be the set of points in $R$ ordered by a total order $\leq_T$ compatible with $\leq$. Then, by Proposition \ref{lemOrder} we have that det$\big(L_{R'}\big)=1$, because the matrix $L_{R'}$ is lower triangular and the elements in the diagonal are all ones.

By definition of $R'$ we can transform the matrix $L_{R'}$ into $L_R$ by interchanging rows and columns, hence $|\mbox{det}\big(L_R\big)|=1$. The result follows by noticing that for any interchange between two columns (that possibly change the sign of the determinant) we have to perform the same interchange between the corresponding columns. Hence det$\big(L_R\big)=1$.
\hfill$\Box$

\vspace{5mm}

From now on, unless otherwise stated,  we will consider $<_T$, the graded reverse lexicographic order, i.e., ${\bf k}<_T{\bf k'}$ if and only if $
|{\bf k}|<|{\bf k'}|$ or $|{\bf k}|=|{\bf k'}|$ and the left-most nonzero entry of ${\bf k'}-{\bf k}$ is negative. By ${\bf k}\leq_T{\bf k'}$ we mean ${\bf k}<_T{\bf k'}$ or ${\bf k}={\bf k'}$.

\vspace{2mm}

By Proposition \ref{lemOrder}, ordering the set $R$ by $<_T$, the  matrix $L_R$ turns out to be a lower-triangular matrix
\[L_R=\left(\begin{array}{cccc}
\binom{\bf k_1}{\bf k_1} & 0 & \cdots & 0\\
\\
\binom{\bf k_2}{\bf k_1} & \binom{\bf k_2}{\bf k_2} & 0 & \cdots\\
\\
\vdots & & \ddots & \\
\\
\binom{\bf k_r}{\bf k_1} & \binom{\bf k_r}{\bf k_2} & \cdots & \binom{\bf k_r}{\bf k_r}\\
\end{array}\right)\]
and hence $U_R$ is an upper-triangular matrix.
\vspace{3mm}

\begin{remark}
For $1\leq i\leq n$ and $r\in\Z_{\geq 0}$, let us denote by $R_{i,r}$  the set of points
\[R_{i,r}=\{\lambda{\bf e_i}\ |\ 0\leq\lambda\leq r\}.\]
Then the Pascal matrices associated to $R_{i,r}$ coincide with the classical Pascal matrices of order $r+1$.
\label{RemPcl}
\end{remark}

\vspace{5mm}

Exactly as in the classical case, we can consider the infinite versions of the Pascal matrices. From now on, the set $R\subseteq\Z^n_{\geq 0}$ is not necessarily finite, unless otherwise stated.

\section{The linear algebra of Pascal matrices}

In the previous section we have generalized the classical Pascal matrices of certain order to Pascal matrices associated with a set of points $R\subseteq\Z_{\geq 0}^n$. A lot of work has been done studying algebraic properties of the classical Pascal matrices (see for instance \cite{BrPi})  and some of its generalizations (see \cite{Z}). In this section we look for analogous properties in the multivariate case.  The first thing to notice is that we need to ask for conditions on the set $R$.
Roughly speaking, the points in $R$ must be {\em together}.
\begin{definition}
Given a monomial ideal $J\subseteq\C[x_1,\ldots,x_n]$, the set of standard monomials, denoted by ${\rm std}_J$, is the set of monomials which do not belong to $J$. We say that a set of points $R$ in $\Z^n$ satisfies the monomial condition if it can be identified with ${\rm std}_J$ for some monomial ideal $J$ in $\C[x_1,\ldots,x_n]$, by means of the identification
\[x_1^{a_1}\cdots x_n^{a_n}\longleftrightarrow (a_1,\ldots,a_n)\]
\label{monomial}
\end{definition}
Note that the set $R$ is finite if and only if the ideal $J$ is zero-dimensional.

\vspace{3mm}

\begin{remark}
Satisfying the monomial condition for a set $R$ is equivalent to any of the following:
\begin{enumerate}
\item[(i)] For any ${\bf k}\in R$,
\[\{{\bf i}\in\Z^n_{\geq 0}\ |\ {\bf i}\leq{\bf k}\}\subseteq R.\]

\item[(ii)] For any ${\bf k},{\bf k'}\in R$ with ${\bf k}\leq{\bf k'}$,
\[\{{\bf i}\in R\ |\ {\bf k'}\leq{\bf i}\leq{\bf k}\}=\{{\bf i}\in\Z^n\ |\ {\bf k'}\leq{\bf i}\leq{\bf k}\}.\]
\end{enumerate}
\label{RemeqMC}
\end{remark}

\vspace{3mm}

If a set $R$ satisfies such condition, we will see that the matrices $L_R$, $U_R$ and $S_R$ inherit many good properties that classical Pascal matrices have.

\vspace{3mm}

Exactly as it happens in the case of classical Pascal matrices, the lower-triangular and upper-triangular matrices give the LU-factorization of the symmetric one. It coincides with the Cholesky factorization since $U_R=L_R^T$.

\begin{lemma}
Let $R$ be a set of points which satisfies the monomial condition. Then
\[S_R=L_RU_R.\]
\label{LU}
\end{lemma}

{\em Proof.} It is a direct consequence of (\ref{eqVand}). \hfill$\Box$

\begin{corollary}
For any finite set of points $R$ satisfying the monomial condition,
\[{\rm det}(S_R)=1.\]
\label{det1}
\end{corollary}

\vspace{2mm}

\begin{example}
The ordered set $R=\{(0,0),(0,1),(1,0),(0,2)\}$ satisfies the monomial condition. Its associated Pascal matrices are
\[L_R=\left(\begin{array}{cccc}
1 & 0 & 0 & 0\\
1 & 1 & 0 & 0\\
1 & 0 & 1 & 0\\
1 & 2 & 0 & 1\\
\end{array}\right),\ U_R=L_R^T\ \mbox{ and }S_R=\left(\begin{array}{cccc}
1 & 1 & 1 & 1\\
1 & 2 & 1 & 3\\
1 & 1 & 2 & 1\\
1 & 3 & 1 & 6\\
\end{array}\right)\]
We can check that $S_R=L_RU_R$ and that det$(S_R)=1$.

Now notice that the set $R'=R\setminus\{(0,1)\}$ does not satisfy the monomial condition and we loose the properties of the Lemma \ref{LU} and Corollary \ref{det1}, since
\[L_R=\left(\begin{array}{ccc}
1 & 0 & 0\\
1 & 1 & 0\\
1 & 0 & 1\\
\end{array}\right)\ \mbox{ and }S_R=\left(\begin{array}{ccc}
1 & 1 & 1\\
1 & 2 & 1\\
1 & 1 & 6\\
\end{array}\right)\]
\label{EjR}
\end{example}

\vspace{3mm}

\begin{lemma}
Let $R$ be a set of points satisfying the monomial condition. Then the inverse of the lower triangular Pascal matrix is
\[L_R^{-1}=\left(\binom{\bf k}{\bf k'}(-1)^{|{\bf k}|-|{\bf k'}|}\right)_{{\bf k},{\bf k'}\in R}\]
Moreover
\[L_R^{-1}=D_RL_RD_R,\]
where $D_R= {\rm diag}\big((-1)^{|{\bf k}|}\big)_{{\bf k}\in R}$.
\label{inv}
\end{lemma}

{\em Proof.}
The first equality is a straightforward consequence of the identity (\ref{eqbinom}). Note how we use the fact that $R$ satisfies the monomial condition, since in this case (see Remark \ref{RemeqMC})
\[\{{\bf i}\in R\ |\ {\bf k'}\leq{\bf i}\leq{\bf k}\}=\{{\bf i}\in\Z^n\ |\ {\bf k'}\leq{\bf i}\leq{\bf k}\}.\]

The second identity is also straightforward taking into account the easy remark that $(-1)^\alpha=(-1)^{-\alpha}$ for any $\alpha$.
\hfill$\Box$

\vspace{3mm}

As a consequence we obtain the inverses of the matrices $U_R$ and $S_R$.

\begin{corollary}
Let $R$ be a set of points satisfying the monomial condition. Then
\[U_R^{-1}=D_R U_R D_R,\]
\[S_R^{-1}=D_R S_R D_R.\]
\label{corinv}
\end{corollary}

{\em Proof.} It follows by Lemma \ref{LU} and Lemma \ref{inv},
taking into account that ${\rm diag}\big((-1)^{|{\bf k}|}\big)_{{\bf k}\in R}^T={\rm diag}\big((-1)^{|{\bf k}|}\big)_{{\bf k}\in R}$  and that ${\rm diag}\big((-1)^{|{\bf k}|}\big){\rm diag}\big((-1)^{|{\bf k}|}\big)=I_r$, where $r$ is the cardinality of the set $R$. \hfill$\Box$

\vspace{3mm}

\begin{remark}
Notice that, since $D_R=D_R^{-1}$, Lemma \ref{inv} and Corollary \ref{corinv} imply that $L_R$, $U_R$ and $S_R$ are similar matrices to their respective inverses.
\label{remSim}
\end{remark}

\vspace{3mm}

\begin{lemma}
Let $R=\{{\bf k_1},\ldots,{\bf k_r}\}$ be a set of points in $\Z^n$ satisfying the monomial condition, and let $p$ be any integer. Then
\[L_R^p\left(\begin{array}{c}
{\bf x}^{\bf k_1}\\
\\
\vdots\\
\\
{\bf x}^{\bf k_r}\\
\end{array}\right)=\left(\begin{array}{c}
(p{\bf 1}+{\bf x})^{\bf k_1}\\
\\
\vdots\\
\\
(p{\bf 1}+{\bf x})^{\bf k_r}\\
\end{array}\right)\]
where we denote $(p{\bf 1}+{\bf x})^{\bf k}=\prod_{j=1}^n(p+x_j)^{k_j}$.
\label{lem1}
\end{lemma}

{\em Proof.} If $p=0$ the claim is obvious. Let us suppose first that $p=1$. The ${\bf k}$-th row of $L_R$ times the vector $\big({\bf x}^{\bf k_1},\ldots,{\bf x}^{\bf k_r}\big)^T$ is
\[\sum_{{\bf k'}\in R}\binom{\bf k}{\bf k'}{\bf x}^{\bf k'}=\sum_{{\bf k'}\in R,{\bf k'}\leq{\bf k}}\binom{\bf k}{\bf k'}{\bf x}^{\bf k'}.\]
Since $R$ satisfies the monomial condition, this is equal to
\[\sum_{{\bf k'}\in\Z^n_{\geq 0},{\bf k'}\leq{\bf k}}\binom{\bf k}{\bf k'}{\bf x}^{\bf k'}=\prod_{j=1}^n\sum_{k_j'=0}^{k_j}\binom{k_j}{k_j'}x_j^{k_j'}=\prod_{j=1}^n(1+x_j)^{k_j}.\]
Suppose the claim is true for $p>0$ and we prove it for $p+1$. We have
\[L_R^{p+1}\left(\begin{array}{c}
{\bf x}^{\bf k_1}\\
\vdots\\
{\bf x}^{\bf k_r}\\
\end{array}\right)=L_R\left(\begin{array}{c}
(p{\bf 1}+{\bf x})^{\bf k_1}\\
\vdots\\
(p{\bf 1}+{\bf x})^{\bf k_r}\\
\end{array}\right)\]
and the ${\bf k_i}$-th component of this vector is
\[\sum_{{\bf k}\in R}\binom{\bf k_i}{\bf k}(p{\bf 1}+{\bf x})^{\bf k}=\sum_{{\bf k}\leq{\bf k_i}}\binom{\bf k_i}{\bf k}(p{\bf 1}+{\bf x})^{\bf k}=\big((p+1){\bf 1}+{\bf x}\big)^{\bf k_i}.\]

For $p<0$ we have to use Lemma \ref{inv} and similar arguments as before.
\hfill$\Box$

\vspace{3mm}

\begin{example}
Consider the set of points $R=\{(0,0),(0,1),(1,0),(2,0)\}$ ordered by $<_T$. It satisfies the monomial condition. The associated lower-triangular Pascal matrix is
\[L_R=\left(\begin{array}{cccc}
1 & 0 & 0 & 0\\
1 & 1 & 0 & 0\\
1 & 0 & 1 & 0\\
1 & 0 & 2 & 1\\
\end{array}\right)\]
and we have
\[\left(\begin{array}{cccc}
1 & 0 & 0 & 0\\
1 & 1 & 0 & 0\\
1 & 0 & 1 & 0\\
1 & 0 & 2 & 1\\
\end{array}\right)\left(\begin{array}{c}
1\\
y\\
x\\
x^2\\
\end{array}\right)=\left(\begin{array}{c}
1\\
1+y\\
1+x\\
(1+x)^2\\
\end{array}\right)\]
If $R$ does not satisfy the monomial condition the result is no longer true, as can be checked with the set $R=\{(0,0),(1,0),(1,1),(2,0)\}$.
\label{Ex2.10}
\end{example}

\vspace{3mm}

\begin{corollary}
For $p\in\Z$ the ${\bf k_i},{\bf k_j}$-th entry of $L_R^p$ is $p^{|{\bf k_i}|-|{\bf k_j}|}\binom{\bf k_i}{\bf k_j}$, i.e.,
\[L_R^p=\left(p^{|{\bf k_i}|-|{\bf k_j}|}\binom{\bf k_i}{\bf k_j}\right)_{{\bf k_i},{\bf k_j}\in R}\]
or, in other words,
\[L_R^p=D_{R,p}L_RD_{R,p}^{-1}\]
where $D_{R,p}={\rm diag}\big(p^{|{\bf k}|}\big)_{{\bf k}\in R}$.

It follows that
\[L_R^p\equiv I_r\mbox{ mod }p\]
where $r$ is the cardinal of $R$ (possibly infinite).
\label{PowersL}
\end{corollary}

\vspace{3mm}

\begin{corollary}
For ${\bf k_i},{\bf k_j}\in\Z^n_{\geq 0}$ with ${\bf k_j}\leq{\bf k_i}$ and $p,q\in\Z$,
\[\sum_{{\bf k_j}\leq{\bf k}\leq{\bf k_i}}p^{|{\bf k_i}|-|{\bf k}|}q^{|{\bf k}|-|{\bf k_j}|}\binom{\bf k_i}{\bf k}\binom{\bf k}{\bf k_j}=(p+q)^{|{\bf k_i}|-|{\bf k_j}|}\binom{\bf k_i}{\bf k_j}.\]
\end{corollary}

{\em Proof.}
It is a direct consequence of $L_R^pL_R^q=L_R^{p+q}$. \hfill$\Box$

\vspace{3mm}

\begin{proposition}
Let $R\subseteq\Z^n_{\geq 0}$ be a set satisfying the monomial condition. The powers of the corresponding Pascal matrix $L_R$ are exponential matrices of the form
\[L_R^p=e^{pA_R}=I_r+pA_R+\frac{p^2}{2!}A_R^2+\cdots,\]
for $p\in\Z$, where $I_r$ is the identity matrix of size $r=|R|$, and $A_R$ is defined as
\[\big(A_R\big)_{{\bf k_i}{\bf k_j}}=\left\{
\begin{array}{cl}
\binom{\bf k_i}{\bf k_j} & \mbox{ if }|{\bf k_i}|=|{\bf k_j}|+1\\
\\
0 & \mbox{ otherwise}\\
\end{array}\right.\]
for ${\bf k_i},{\bf k_j}\in R$.
\end{proposition}

{\em Proof.} Exponential matrices arise naturally as solutions of systems of differential equations. Let
\[{\bf y}(t)=\left(\begin{array}{c}
y_{\bf k_1}(t)\\
\vdots\\
y_{\bf k_j}(t)\\
\vdots\\
\end{array}\right)_{{\bf k_j}\in R}\]
be a vector of functions (as always indexed by the elements in $R$), and consider the system of equations
\begin{equation}
\begin{array}{l}
\frac{d{\bf y}(t)}{dt}=A_R{\bf y}(t),\\
\\
{\bf y}(0)=(y_{\bf k_1}(0),\ldots,y_{\bf k_j}(0),\ldots).\\
\end{array}
\label{system}
\end{equation}
The unique solution to this system is
\[{\bf y}(t)=e^{A_Rt}{\bf y}(0).\]
Let us define the matrix
\[L_R(t)=\left(t^{|{\bf k_i}|-|{\bf k_j}|}\binom{\bf k_i}{\bf k_j}\right)_{{\bf k_i},{\bf k_j}\in R}\]
We claim that $L_R(t){\bf y}(0)$ is solution to the system (\ref{system}). Indeed, if ${\bf y}(t)=L_R(t){\bf y}(0)$, then, for any ${\bf k_i}\in R$, the ${\bf k_i}$-coordinate of the vector of functions is
\[y_{\bf k_i}(t)=\sum_{{\bf k}\in R}t^{|{\bf k_i}|-|{\bf k}|}\binom{\bf k_i}{\bf k}y_{\bf k}(0)\]
Then, the ${\bf k_i}$-coordinate of $A_R{\bf y}(t)$, denoted $a_{\bf k_i}$, is
\[a_{\bf k_i}=\sum_{{\bf k}\in R,\ |{\bf k_i}|=|{\bf k}|+1}\binom{\bf k_i}{\bf k}\sum_{{\bf k'}\in R}t^{|{\bf k}|-|{\bf k'}|}\binom{\bf k}{\bf k'}y_{\bf k'}(0)\]
Notice that if $|{\bf k_i}|=|{\bf k}|+1$ and $\binom{\bf k_i}{\bf k}\neq 0$, then
\[{\bf k}={\bf k_i}-{\bf e_j},\ \mbox{ for }j=1,\ldots,n\]
and since $R$ satisfies the monomial condition all such ${\bf k}$ belongs to $R$. Then
\[\begin{array}{ll}
a_{\bf k_i} & = \sum_{j=1}^n\sum_{{\bf k'}\in R}\binom{\bf k_i}{{\bf k_i}-{\bf e_j}}\binom{{\bf k_i}-{\bf e_j}}{\bf k'}t^{|{\bf k_i}|-1-|{\bf k'}|}y_{\bf k'}(0)\\
\\
 & =\sum_{{\bf k'}\in R}t^{|{\bf k_i}|-|{\bf k'}|-1}y_{\bf k'}(0)\sum_{j=1}^n\binom{\bf k_i}{{\bf k_i}-{\bf e_j}}\binom{{\bf k_i}-{\bf e_j}}{\bf k'}\\
 \end{array}\]
Denoting ${\bf k_i}=(k_i^{(1)},\ldots,k_i^{(n)})$, we have
\[\sum_{j=1}^n\binom{\bf k_i}{{\bf k_i}-{\bf e_j}}\binom{{\bf k_i}-{\bf e_j}}{\bf k'}=k_i^{(1)}\binom{{\bf k_i}-{\bf e_1}}{\bf k'}+\cdots+k_i^{(n)}\binom{{\bf k_i}-{\bf e_n}}{\bf k'}=\]
\[=\frac{{\bf k_i}!}{{\bf k'}!}\Big(\frac{1}{({\bf k_i}-{\bf e_1}-{\bf k'})!}+\cdots+\frac{1}{({\bf k_i}-{\bf e_n}-{\bf k'})!}\Big)=\binom{\bf k_i}{\bf k'}\big(|{\bf k_i}|-|{\bf k'}|\big)\]
Therefore we have proved that
\[a_{\bf k_i}=\frac{d y_{\bf k_i}}{dt},\]
or, in other words, $L_R(t){\bf y}(0)$ is a solution to the system (\ref{system}).

By Corollary \ref{PowersL} we have that $L_R(p)=L_R^p$ and we are done.
\hfill$\Box$

\vspace{3mm}

\begin{remark}
If $R$ is a finite set, then $A_R^\ell$ is the zero matrix for $\ell\geq r=|R|$.
\end{remark}

\vspace{3mm}

\begin{example}
Let $R=\{(0,0),(0,1),(1,0),(1,1),(2,0)\}$. Then
\[A_R=\left(\begin{array}{ccccc}
0 & 0 & 0 & 0 & 0\\
1 & 0 & 0 & 0 & 0\\
1 & 0 & 0 & 0 & 0\\
0 & 1 & 1 & 0 & 0\\
0 & 0 & 2 & 0 & 0\\
\end{array}\right)\]
Notice that in this example $A_R^\ell$ is zero for $\ell\geq 3$.
\end{example}

\vspace{3mm}

The matrix $A_R$ does not seem to satisfy the properties of its classical analogue, the so-called {\em creation matrix} or {\em derivation matrix} (see \cite{AcetoTrigiante}, where the authors use them to define the Pascal matrices).

\vspace{5mm}

Exactly as the identity in Lemma \ref{anBin} (i) is the key for Lemma \ref{lem1}, we can use Lemma \ref{anBin} (ii) to derive more identities, but in this case we need to deal with infinite matrices.

\begin{definition}
By $L$, $U$ and $S$ we denote the corresponding infinite matrices associated with the set $\Z^n_{\geq 0}$ (ordered by the total order $<_T$).
\end{definition}

\begin{remark}
Recall that a set $R$ satisfying the monomial condition is not necessarily finite. In particular the set $R=\Z^n_{\geq 0}$ satisfies the monomial condition. Hence the results above hold for  the matrices $L$, $U$ and $S$.
\label{inf}
\end{remark}

\vspace{3mm}

\begin{lemma}
\[U\left(\begin{array}{c}
\\
\vdots\\
\\
{\bf x}^{\bf k}\\
\\
\vdots\\
\\
\end{array}\right)=\left(\begin{array}{c}
\\
\vdots\\
\\
\frac{{\bf x}^{\bf k}}{\big({\bf 1}-{\bf x}\big)^{{\bf k}+{\bf 1}}}\\
\\
\vdots\\
\\
\end{array}\right)\]
where $\big({\bf 1}-{\bf x}\big)^{{\bf k}+{\bf 1}}=\prod_{j=1}^n(1-x_j)^{k_j+1}$.
\label{lem2}
\end{lemma}

{\em Proof.}
Let us denote by $u_{{\bf k}{\bf k'}}$ the entries of the matrix $U$. For any ${\bf k}\in\Z^n_{\geq 0}$,
\[\begin{array}{ll}
\sum_{{\bf k'}\in\Z^n_{\geq 0}}u_{{\bf k}{\bf k'}}{\bf x}^{\bf k'} & =\sum_{{\bf k'}\geq_T{\bf k}}\binom{\bf k'}{\bf k}{\bf x}^{\bf k'}\\
\\
 & ={\bf x}^{\bf k}\sum_{{\bf k'}\in\Z^n_{\geq 0}}\binom{{\bf k'}+{\bf k}}{\bf k}{\bf x}^{\bf k'}\\
 \end{array}\]
and the result follows by Lemma \ref{anBin} (ii). Notice that we use the easy remark that $\{{\bf k}\geq_T{\bf 0}\}=\Z^n_{\geq 0}$.
\hfill$\Box$

\vspace{3mm}

\begin{lemma}
\[U\left(\begin{array}{c}
\vdots\\
\frac{{\bf x}^{\bf k}}{{\bf k}!}\\
\vdots\\
\end{array}\right)=e^{x_1+\cdots+x_n}\left(\begin{array}{c}
\vdots\\
\frac{{\bf x}^{\bf k}}{{\bf k}!}\\
\vdots\\
\end{array}\right)\]
\end{lemma}

{\em Proof.}
For any ${\bf k}\in\Z_{\geq 0}^n$, the ${\bf k}$-row of $U$ times $\big(\ldots,\frac{{\bf x}^{\bf k}}{{\bf k}!},\ldots\big)$ gives
\[\begin{array}{ll}
\sum_{{\bf k'}\in\Z_{\geq 0}^n}\binom{{\bf k'}}{\bf k}\frac{{\bf x}^{\bf k'}}{{\bf k'}!} & =\sum_{{\bf k'}\geq{\bf k}}\binom{{\bf k'}}{\bf k}\frac{{\bf x}^{\bf k'}}{{\bf k'}!}\\
 & =\frac{{\bf x}^{\bf k}}{{\bf k}!}\sum_{{\bf k'}\geq{\bf k}}\frac{1}{({\bf k'}-{\bf k})!}{\bf x}^{{\bf k'}-{\bf k}}\\
 & =\frac{{\bf x}^{\bf k}}{{\bf k}!}e^{x_1+\cdots+x_n}.\\
 \end{array}\]
\hfill$\Box$

\vspace{3mm}

\begin{lemma}
\[S\left(\begin{array}{c}
\vdots\\
\\
{\bf x}^{\bf k}\\
\\
\vdots\\
\end{array}\right)=\left(\begin{array}{c}
\vdots\\
\\
\frac{1}{\big({\bf 1}-{\bf x}\big)^{{\bf k}+{\bf 1}}}\\
\\
\vdots\\
\end{array}\right)\]
where $\big({\bf 1}-{\bf x}\big)^{\bf k}=(1-x_1)^{k_1}\cdots(1-x_n)^{k_n}$ for ${\bf k}=(k_1,\ldots,k_n)\in\Z^n_{\geq 0}$.
\label{lem3}
\end{lemma}

{\em Proof.} By Lemma \ref{LU} and Lemma \ref{lem2} with $R=\Z^n_{\geq 0}$, we have
\[\begin{array}{ll}
\sum_{{\bf k'}\in\Z^n_{\geq 0}}s_{{\bf k}{\bf k'}}{\bf x}^{\bf k'} & =\sum_{{\bf k'}\in\Z^n_{\geq 0}}\ell_{{\bf k}{\bf k'}}\frac{{\bf x}^{\bf k'}}{({\bf 1}-{\bf x})^{{\bf k'}+{\bf 1}}}\\
\\
 & =\sum_{{\bf k'}\leq{\bf k}}\binom{\bf k}{\bf k'}\frac{{\bf x}^{\bf k'}}{({\bf 1}-{\bf x})^{{\bf k'}+{\bf 1}}}\\
 \\
 & =\frac{1}{{\bf 1}-{\bf x}}\sum_{{\bf k'}\leq{\bf k}}\binom{\bf k}{\bf k'}\left(\frac{\bf x}{{\bf 1}-{\bf x}}\right)^{{\bf k'}}\\
\end{array}\]
and the result follows by Lemma \ref{anBin} (i).
\hfill$\Box$

\subsection{Binomial transform of sequences}

If we consider sequences $\{a_{\bf k}\}_{\bf k}$ depending on parameter vectors ${\bf k}\in\Z^n_{\geq 0}$, then the multidimensional binomial transform can be defined as
\[b_{\bf k}:=\sum_{{\bf i}\leq{\bf k}}\binom{\bf k}{\bf i}a_{\bf i}.\]

Considering the sequences $\{a_{\bf k}\}$ and $\{b_{\bf k}\}$ as column vectors $A$ and $B$ (ordered by the graded reverse lexicographic order $<_T$), we can write the binomial transform in matrix terms as
\begin{equation}
LA=B,
\label{BinTMat}
\end{equation}
where $L$ is the infinite lower triangular Pascal matrix. Notice that in Lemma \ref{lem1} we have proved that the sequence $\{\big({\bf 1}+{\bf x}\big)^{\bf k}\}_{\bf k}$ is the binomial transform of the sequence $\{{\bf x}^{\bf k}\}_{\bf k}$. More generally, for $p\in\Z$, the binomial transform of the sequence $\{\big((p-1){\bf 1}+{\bf x}\big)^{\bf k}\}_{\bf k}$ is the sequence $\{\big(p{\bf 1}+{\bf x}\big)^{\bf k}\}_{\bf k}$.

\vspace{2mm}

Notice that, since both $L$ and $L_R$ are lower-triangular, the {\em truncation} of (\ref{BinTMat}) also holds:
\begin{equation}
L_R\left(\begin{array}{c}
\vdots\\
a_{\bf k}\\
\vdots\\
\end{array}\right)_{{\bf k}\in R}=\left(\begin{array}{c}
\vdots\\
b_{\bf k}\\
\vdots\\
\end{array}\right)_{{\bf k}\in R}
\label{BinTMatTr}
\end{equation}
for any set $R\subseteq\Z^n_{\geq 0}$ satisfying the monomial condition.

\subsection{Relation with Stirling and Vandermonde matrices}

The Stirling numbers of second kind $S(n,k)$ are well known combinatorial numbers (see for instance \cite{Boy} for an introduction on these combinatorial numbers) that can be defined as
\[x^k=\sum_{n=0}^\infty S(n,k)n!\binom{x}{n}\]
This identity could be seen as a first step of a tight relation between binomial coefficients and Stirling numbers when we give to $x$ a positive integer value.
Indeed, if we define the factorial Stirling matrices as
\[\bar S_n=\Big(i!S(i,j)\Big)_{0\leq i,j\leq n},\]
the equations above for $x=0,1,\ldots, n$ can be written in matrix form as follows:
\begin{equation}
L_n\bar S_n=\left(\begin{array}{ccccc}
1 & 0 & 0 & \cdots & 0\\
1 & 1 & 1 & \cdots & 1\\
1 & 2 & 2^2 & \cdots & 2^n\\
\\
 & & & \cdots & \\
 \\
 1 & n & n^2 & \cdots & n^n\\
 \end{array}\right)
\label{firstRel}
\end{equation}
where $L_n$ denotes the classical lower-triangular Pascal matrix of order $n+1$. Notice that the third matrix is of Vandermonde type.

Many relations between the Pascal matrices and (factorial) Stirling matrices together with  Vandermonde matrices have been found in the classical situation. We cite Theorem 2.4 in \cite{CheonKim1} or Theorem 2.1 in \cite{CheonKim2}. See also \cite{El-M}, \cite{MikCheon} and \cite{YY}.

\vspace{2mm}

Using binomial transform of sequences we will generalize the factorization in (\ref{firstRel}), relating the multivariate Pascal matrix with a generalized Stirling matrix whose entries are generalizations of Stirling numbers, namely the Stirling polynomials of second kind, as defined in \cite{Cobo1}.

\begin{definition}
Given ${\bf k}\in\Z_{\geq 0}^n$ and $\ell\in\Z_{\geq 0}$, the Stirling polynomials of second kind $S_{\bf k}^{(\ell)}(x_0,x_1,\ldots,x_n)$ are defined by the generating function
\[\frac{1}{{\bf k}!}e^{x_0t}\prod_{j=1}^n\big(e^{x_jt}-1\big)^{k_j}=\sum_{\ell=0}^\infty S_{\bf k}^{(\ell)}(x_0,x_1,\ldots,x_n)\frac{t^\ell}{\ell!}.\]
\end{definition}
These polynomials appear naturally in the Weyl algebra, since for any $(\alpha_0,\ldots,\alpha_n)\in\C^{n+1}$ and any $\ell\in\Z_{\geq 0}$ we have
\[\big(\alpha_0+\alpha_1x_1\partial_1+\cdots+\alpha_nx_n\partial_n\big)^\ell=\sum_{{\bf k}\in\Z^n_{\geq 0},\ |{\bf k}|\leq\ell}S_{\bf k}^{(\ell)}(\alpha_0,\ldots,\alpha_n)x_1^{k_1}\partial_1^{k_1}\cdots x_n^{k_n}\partial_n^{k_n}.\]
The Stirling numbers of second kind are a specialization of the Stirling polynomials:
\[S_{k{\bf e_i}}^{(\ell)}\big(x_0=0,x_i=1\big)=S(\ell,k).\]

A closed formula for the Stirling polynomials in terms of the Stirling numbers is the following (see \cite{Cobo1})
\begin{equation}
S_{\bf k}^{(\ell)}({\bf x})=\sum_{{\bf i}\geq {\bf k},\ |{\bf i}|\leq\ell}\binom{\ell}{{\bf i}}\Big(\prod_{j=1}^n S(i_j,k_j)\Big)x_0^{\ell-|{\bf i}|}x_1^{i_1}\cdots x_n^{i_n}\in\Z[x_0,\ldots,x_n].
\label{C2def}
\end{equation}

\begin{definition}
For any ${\bf k}\in\Z^n_{\geq 0}$ we define the linear form
\[A_{\bf k}=x_0+k_1x_1+\cdots+k_nx_n.\]

\end{definition}

\begin{proposition} (Proposition 1 in \cite{Cobo1})
 Let $\ell$ be a positive integer. The polynomial sequence $\{A_{\bf k}^\ell\}_{\bf k}$ is the binomial transform of the polynomial sequence $\{{\bf k}!S_{\bf k}^{(\ell)}\}_{\bf k}$.
\label{BinSeq}
\end{proposition}

\begin{definition}
Given an ordered set of points $R\subseteq\Z^n_{\geq 0}$ and given $\ell\in\Z_{\geq 0}$, we define the following matrices of size $|R|\times(\ell+1)$:
\begin{enumerate}
\item[(i)] The Vandermonde matrix
\[V_{R,\ell}=\left(A_{\bf k}^j\right)_{{\bf k}\in R,\ 0\leq j\leq\ell}.\]

\item[(ii)] The generalized factorial Stirling matrix
\[\mathcal S_{R,\ell}=\left({\bf k}!S_{\bf k}^{(j)}\right)_{{\bf k}\in R,\ 0\leq j\leq\ell}.\]
\end{enumerate}
In both cases the rows of the matrix are ordered as the elements in $R$, while the columns are ordered by non-negative integers $\leq\ell$.
\end{definition}

\begin{theorem}
Let $R$ be a set of points satisfying the monomial condition, and let $\ell$ be a positive integer. Then
\[L_R \mathcal S_{R,\ell}=V_{R,\ell}.\]
\label{Decomp}
\end{theorem}

{\em Proof.} By Proposition \ref{BinSeq} we have
$$A_{\bf k}^\ell=\sum_{{\bf i}\leq {\bf k}}\binom{\bf k}{\bf i}{\bf i}!S_{\bf i}^{(\ell)},$$
which proves the claim. As usual we use here that the order $<_T$ is compatible with the partial order $\leq$.
\hfill$\Box$

\vspace{3mm}

\begin{example}
Let $L_R$ be the matrix given in Example \ref{Ex2.10}, and let $\ell=5$.
Using the closed formula (\ref{C2def}) the reader can check that the generalized factorial Stirling matrix is
\[\mathcal S_{R,5}=\left(\begin{array}{cccccc}
1 & x_0 & x_0^2 & x_0^3 & x_0^4 & x_0^5\\
\\
0 & x_2 & 2x_0x_2+x_2^2 & \begin{array}{c}
                          3x_0^2x_2+\\
                          3x_0x_2^2+x_2^3\\
                          \end{array} & \begin{array}{c}
                                         4x_0^3x_2+6x_0^2x_2^2\\
                                         +4x_0x_2^3+x_2^4\\
                                         \end{array}  & \begin{array}{c}
                                                         5x_0^4x_2+10x_0^3x_2^2\\
                                                         +10x_0^2x_2^3+\\
                                                         5x_0x_2^4+x_2^5\\
                                                         \end{array}\\
\\
0 & x_1 & 2x_0x_1+x_1^2 & \begin{array}{c}
                          3x_0^2x_1+\\
                          3x_0x_1^2+x_1^3\\
                          \end{array} & \begin{array}{c}
                                         4x_0^3x_1+6x_0^2x_1^2\\
                                         +4x_0x_1^3+x_1^4\\
                                         \end{array}  & \begin{array}{c}
                                                         5x_0^4x_1+10x_0^3x_1^2\\
                                                         +10x_0^2x_1^3+\\
                                                         5x_0x_1^4+x_1^5\\
                                                         \end{array}\\
\\
0 & 0 & 2x_1^2 & 6x_0x_1^2+6x_1^3 & \begin{array}{c}
                                     12x_0^2x_1^2+\\
                                     24x_0x_1^3+\\
                                     +14x_1^4\\
                                     \end{array} & \begin{array}{c}
                                                    20x_0^3x_1^2\\
                                                    +60x_0^2x_1^3\\
                                                    +70x_0x_1^4\\
                                                    +30x_1^5\\
                                                    \end{array}\\
\end{array}\right)\]
and that we have
\[L_R\mathcal S_{R,5}=\left(\begin{array}{cccccc}
1 & x_0 & x_0^2 & x_0^3 & x_0^4 & x_0^5\\
1 & x_0+x_2 & (x_0+x_2)^2 & (x_0+x_2)^3 & (x_0+x_2)^4 & (x_0+x_2)^5\\
1 & x_0+x_1 & (x_0+x_1)^2 & (x_0+x_1)^3 & (x_0+x_1)^4 & (x_0+x_1)^5\\
1 & x_0+2x_1 & (x_0+2x_1)^2 & (x_0+2x_1)^3 & (x_0+2x_1)^4 & (x_0+2x_1)^5\\
\end{array}\right)\]
\end{example}

\vspace{3mm}

\begin{remark}
In fact, we get results of the type of Theorem \ref{Decomp} whenever we have two parametric sequences of the form
\[\{a_{{\bf k},\ell}\}\ \mbox{ and }\ \{b_{{\bf k},\ell}\},\]
with ${\bf k}\in\Z^n_{\geq 0}$ and $\ell\in\Z_{\geq 0}$ satisfying the property that for any $\ell\in\Z_{>0}$, the sequence $\{b_{{\bf k},\ell}\}$ is the binomial transform of $\{a_{{\bf k},\ell}\}$. Then, defining the infinite matrices
\[A=\Big(a_{{\bf k},\ell}\Big)\ \mbox{ and }\ B=\Big(b_{{\bf k},\ell}\Big),\]
where the rows of the matrices are indexed by ${\bf k}\in\Z^n_{\geq 0}$ and the columns by $\ell\in\Z_{\geq 0}$, we have the factorization
\[LA=B,\]
or its corresponding truncated version
\[L_RA_{R,\ell}=B_{R,\ell},\]
for any set $R$ satisfying the monomial condition.
\end{remark}

\subsection{The multivariate Riordan group}

 It is a well known fact that the infinite lower-triangular (classical) Pascal matrix is an  element of the Riordan group. In fact the concept of Riordan group, or more precisely, the elements of the Riordan group (also known as Riordan arrays), were introduced as a generalization of the Pascal matrix (see \cite{Riordan} and \cite{RiordanSurvey}).

They are infinite lower-triangular matrices defined in terms of two formal power series $h(x),d(x)\in\C[[x]]$, and denoted by
\[\mathcal R\big(h(x),d(x)\big).\]
\begin{example}
The classical Pascal matrix (see Definition \ref{Pascalcl}) can be seen as an element of the Riordan group, namely
\[\mathcal R\left(\frac{1}{1-x},\frac{x}{1-x}\right).\]
\end{example}

\vspace{3mm}

Riordan arrays can equivalently be defined in terms of the so-called $A$-sequences. With this characterization it is easy to prove that the (multivariate) Pascal matrix $L$ is not an element of the (classical) Riordan group, since it does not exist an $A$-sequence for such a matrix.

\vspace{2mm}

In \cite{WW} the authors generalize the concept of Riordan arrays. They define the Riordan arrays with respect to a sequence $\{c_n\}_{n\in\N}$, so that the classical ones are those corresponding to the sequence $\{c_n=1\}_{n\in\N}$. In Theorem 5.1 in \cite{WW} the generalized Riordan arrays are characterized in terms of the classical ones, and it follows that the Pascal matrix $L$ is not a generalized Riordan array.

\vspace{2mm}

We prove here that the Pascal matrix $L$ is an element of another generalization of the notion of Riordan group, the multivariate Riordan group, introduced in \cite{CheonHuangKim1}. There the authors define the concept of Riordan basis $(G,{\bf X})$, where $G$ is an invertible power series and ${\bf X}$ is a set of variables. Defining a product for Riordan bases they prove that the set of Riordan bases has the structure of a group.

More precisely (see \cite{CheonHuangKim1} for the details), let $K[[Z_1,\ldots,Z_n]]$ be the power series ring and let $\texttt{m}=\big(Z_1,\ldots,Z_n\big)$ be its maximal ideal. Given power series $Y_1,\ldots,Y_n\in\texttt{m}$, we say that they form a set of variables if
\[{\rm det }\Big(\frac{\partial Y_j}{\partial Z_i}({\bf 0})\Big)_{i,j}\neq 0\]
Moreover, there exists a unique $K$-algebra endomorphism
\[\begin{array}{ccc}
K[[Z_1,\ldots,Z_n]] & \longrightarrow & K[[Y_1,\ldots,Y_n]]\\
\\
Z_i & \mapsto & Y_i\\
\end{array}\]
Hence, given $G=\sum_{\bf i}a_{\bf i}{\bf Z}^{\bf i}\in K[[Z_1,\ldots,Z_n]]$, we denote by $G({\bf Y})$ the image of $G$ under this endomorphism.

If $Y_1,\ldots,Y_n$ are variables, the endomorphism is an isomorphism, and for another set of variables $X_1,\ldots,X_n$,
\[{\bf X}\big({\bf Y}\big)=\big(X_1({\bf Y}),\ldots,X_n({\bf Y})\big)\]
is also a set of variables. Moreover, we denote by $\bar{\bf X}$ the preimage of ${\bf X}$ under the $K$-algebra isomorphism ${\bf Z}\mapsto{\bf X}$.

Given $n$-tuples $\big(G,{\bf Y}\big)$ of power series, where $Y_1,\ldots,Y_n$ are variables, we define the product
\begin{equation}
\big(G,{\bf X}\big)\star\big(H,{\bf Y}\big)=\big(GH({\bf X}),{\bf Y}({\bf X})\big)
\label{eqProd}
\end{equation}
If $G$ is invertible, $\big(G,{\bf X}\big)$ has an inverse
\begin{equation}
\big(\frac{1}{G(\bar{\bf X})},\bar{\bf X}\big)
\label{eqInv}
\end{equation}
Thus, the set of $\big(G,{\bf X}\big)$ with $G$ invertible and ${\bf X}$ a set of variables forms a group called the Riordan group (see Corollary 3.2 in \cite{CheonHuangKim1}).

Riordan arrays appear then as infinite matrices in a representation of this group, and will be denoted by $\mathcal R\big(G,X_1,\ldots,X_n\big)$ or simply $\mathcal R\big(G,{\bf X}\big)$,
in analogy with the classical Riordan arrays. More precisely, the relation between the matrix $\mathcal R\big(G,{\bf X}\big)=\big(a_{{\bf i}{\bf j}}\big)$ and the $n$-tuple $(G,{\bf X})$ is give by
\begin{equation}
G{\bf X}^j=\sum_{\bf i}a_{{\bf i}{\bf j}}{\bf Z}^{\bf i}
\label{eqMat}
\end{equation}

\vspace{2mm}

\begin{example}
Let $G(z_1,z_2)=\frac{1}{(1-z_1)(1-z_2)}$ an invertible power series, and $X_1(z_1,z_2)=\frac{z_1}{1-z_1}$ and $X_2(z_1,z_2)=\frac{z_2}{1-z_2}$. We have that
\[G({\bf z})=1+z_1+z_2+z_1^2+z_1z_2+z_2^2+\cdots\]
and that
\[\begin{array}{l}
X_1({\bf z})=z_1+z_1^2+z_1^3+z_1^4+\cdots\\
\\
X_2({\bf z})=z_2+z_2^2+z_2^3+z_2^4+\cdots\\
\end{array}\]
is a set of variables.

Then the matrix $\mathcal R\big(G({\bf z}),X_1({\bf z}),X_2({\bf z})\big)$ looks like
\[\left(\begin{array}{ccccccc}
1 & 0 & 0 & 0 & 0 & 0 & \cdots\\
1 & 1 & 0 & 0 & 0 & 0 & \cdots\\
1 & 0 & 1 & 0 & 0 & 0 & \cdots\\
1 & 2 & 0 & 1 & 0 & 0 & \cdots\\
1 & 1 & 1 & 0 & 1 & 0 & \cdots\\
1 & 0 & 2 & 0 & 0 & 1 & \cdots\\
1 & 3 & 0 & 3 & 0 & 0 & \cdots\\
1 & 1 & 1 & 1 & 2 & 0 & \cdots\\
\vdots & & & \vdots & & & \vdots\\
\end{array}\right)\]
where, according to (\ref{eqMat}), in the columns of the matrix appear the coefficients of the expansions of $G({\bf z}){X_1({\bf z})}^i{X_2({\bf z})}^j$ for $(i,j)\in\Z^2_{\geq 0}$. Recall that the graded reverse lexicographic order $<_T$ orders $\Z^2_{\geq 0}$ as
\[\Z^2_{\geq 0}=\{(0,0),(0,1),(1,0),(0,2),(1,1),(2,0),(0,3),(1,2),(2,1),\ldots\}\]

\end{example}

\vspace{3mm}

As pointed out already in \cite{CheonHuangKim1}, the (multivariate) Riordan matrices are not lower-triangular in general, but let us say block-wise lower-triangular matrices (once we use a total order $<_T$ on $\Z^n_{\geq 0}$ compatible with the partial order $\leq$).

\vspace{3mm}

\begin{remark}
The classical Riordan group is the univariate version of the multivariate Riordan group.
\end{remark}

\vspace{3mm}

Next we prove that the multivariate Pascal matrix $L$ belongs to the multivariate Riordan group.

\begin{proposition}
The powers of the multivariate Pascal matrix $L$ are all elements of the multivariate Riordan group. More precisely,
$$L^p=\mathcal R\big(G({\bf z}),X_1({\bf z}),\ldots,X_n({\bf z})\big)$$
 where
\[\begin{array}{l}
G({\bf z})=\frac{1}{\prod_{j=1}^n(1-pz_j)}\\
\\
X_i({\bf z})=\frac{z_i}{1-pz_i}\ \mbox{ for }1\leq i\leq n\\
\end{array}\]
and $p$ is any integer.
\label{PPascal}
\end{proposition}

{\em Proof.} First notice that $G({\bf z})$ is an invertible power series in $K[[z_1,\ldots,z_n]]$ and that $\{X_1,\ldots,X_n\}$ is a set of variables.

The claim is a consequence of the representation of elements in the multivariate Riordan group given in (\ref{eqMat}) together with the equality in Lemma \ref{anBin} (ii). Indeed, by Corollary \ref{PowersL} applied to the set $R=\Z^n_{\geq 0}$, we have that for any ${\bf k'}\in\Z^n_{\geq 0}$ the generating function of the ${\bf k'}$-column of $L^p$ is $\sum_{{\bf k}\in\Z^n_{\geq 0}}p^{|{\bf k}|-|{\bf k'}|}\binom{\bf k}{\bf k'}{\bf x}^{\bf k}$. We have
\[\begin{array}{ll}
\sum_{{\bf k}\in\Z^n_{\geq 0}}p^{|{\bf k}|-|{\bf k'}|}\binom{\bf k}{\bf k'}{\bf x}^{\bf k} & ={\bf x}^{\bf k'}\sum_{{\bf k}\geq{\bf k'}}\binom{\bf k}{\bf k'}(p{\bf x})^{{\bf k}-{\bf k'}}\\
\\
 & =\frac{1}{1-px_1}\cdots\frac{1}{1-px_n}\left(\frac{x_1}{1-px_1}\right)^{k_1'}\cdots\left(\frac{x_n}{1-px_n}\right)^{k_n'}\\
\\
 &= G({\bf z}){\bf X}^{\bf k'}\\
 \end{array}\]
 as we wanted to prove.
\hfill$\Box$

\vspace{3mm}

\begin{example}
By Proposition \ref{PPascal} the Pascal matrix $L$ is the matrix representing the Riordan basis $\big(G,{\bf X}\big)$, where $G=\frac{1}{\prod_{j=1}^n(1-z_j)}$ and ${\bf X}={\bf X_1}\cdots{\bf X_n}$ with ${\bf X_i}=\frac{z_i}{1-z_i}$. By (\ref{eqInv}) its inverse is
\[\big(G,{\bf X}\big)^{-1}=\Big(\frac{1}{G(\bar {\bf X})},\bar{\bf X}\Big)\]
where $\bar{\bf X}$ is the compositional inverse of ${\bf X}$. Hence we deduce the inverse of the Pascal matrix is
$$L^{-1}=\mathcal R\Big(\frac{1}{\prod_{j=1}^n(1+z_j)},\frac{z_1}{1+z_1},\ldots,\frac{z_n}{1+z_n}\Big)$$
since the compositional inverse of $\frac{z_i}{1-z_i}$ is $\frac{z_i}{1+z_i}$.

Of course, this agrees with the description of $L_R^{-1}$ given in Lemma \ref{inv} for $R=\Z^n_{\geq 0}$.
\end{example}

\vspace{5mm}

\begin{remark}
Notice that the matrices $L_R$ associated with sets of the form
\[R=\left\langle{\bf e_{i_1}},\ldots,{\bf e_{i_r}}\right\rangle\]
are elements of the multivariate Riordan group over $\C[[x_1,\ldots,x_r]]$.
\end{remark}

\end{document}